\pdfoutput=1
\documentclass{amsart}[12pt]
\usepackage{amssymb}
\usepackage{graphicx,epsfig}

\newtheorem{theorem}{Theorem}[section]
\newtheorem{lemma}[theorem]{Lemma}

\theoremstyle{definition}
\newtheorem{definition}[theorem]{Definition}
\newtheorem{example}[theorem]{Example}

\theoremstyle{remark}

\newtheorem{Definition}{\bf Definition}[section]

\newtheorem{Lem}[Definition]{\bf Lemma}

\newtheorem{Note}[Definition]{\bf Note}

\numberwithin{equation}{section}

\newcommand{\abs}[1]{\lvert#1\rvert}

\reversemarginpar

\newcommand{\no}{\noindent}
\newcommand{\realpart}{\mathop{\rm Re}\nolimits}

\newcommand{\ba}{\begin{eqnarray}}
\newcommand{\ea}{\end{eqnarray}}

\newcommand{\allR}{\mathbb{R}}
\newcommand{\allC}{\mathbb{C}}

\newcommand{\allZ}{\mathbb{Z}}

\newcommand{\eps}{\varepsilon}
\newcommand{\opR}{\hat{R}}
\newcommand{\TOp}{\hat{\mathcal{O}}}
\newcommand{\TOpp}{\hat{\mathcal{O}}'}
\newcommand{\NI}{I}
\newcommand{\NV}{V}
\newcommand{\NL}{L}
\newcommand{\T}{\mathcal{T}}
\newcommand{\Tbar}{\overline{\mathcal{T}}}
\newcommand{\updowncontour}{{\downarrow \! \uparrow}}
\newcommand{\MatsubaraSums}{\mathfrak{M}}
\DeclareMathOperator{\nbe}{n_b}
\DeclareMathOperator{\Heaviside}{\mathcal{\vartheta}}
\DeclareMathOperator{\sign}{\epsilon}
\DeclareMathOperator{\Res}{Res}

\begin{document}

\title[On the evaluation of Matsubara sums] {On the evaluation of
Matsubara sums}

\author{Olivier Espinosa}
\address{Departamento de F{\'\i}sica,
Universidad T{\'{e}}c. Federico Santa Mar{\'\i}a, Valpara{\'\i}so, Chile}
\email{olivier.espinosa@usm.cl}



\date{\today}


\begin{abstract}
Given a connected (multi)graph $G$, consisting of $\NV$ vertices and $\NI$ lines,
we consider a class of multidimensional sums of the general form
\begin{equation*}
S_G:=
\sum\limits_{n_1  =  - \infty }^\infty
\sum\limits_{n_2  =  - \infty }^\infty
\cdots \sum\limits_{n_I  =  - \infty }^\infty  
{\frac{{\delta_G (n_1 ,n_2 , \ldots ,n_I;\{N_v\} )}}
{{\left( {n_1^2  + q_1^2 } \right)\left( {n_2^2  + q_2^2 } \right) \cdots \left( {n_I^2  + q_I^2 } \right)}}},
\end{equation*}
where the variables $q_i$ ($i=1,\ldots,\NI$) are real and positive and the variables $N_v$ ($v=1,\ldots,\NV$) are integer-valued. $\delta_G(n_1 ,n_2 , \ldots ,n_I;\{N_v\}  )$ is a function valued in $\{0,1\}$ which imposes a series of linear constraints among the summation variables $n_i$, determined by the topology of the graph $G$.

We prove that these sums, which we call \emph{Matsubara sums}, can be explicitly evaluated by applying a $G$-dependent linear operator $\TOp_G$ to the evaluation of the integral obtained from $S_G$ by replacing the discrete variables $n_i$ by continuous real variables $x_i$ and replacing the sums by integrals.
\end{abstract}

\maketitle

\section{Introduction}
\label{sec-introduction}

Infinite series are ubiquitous in mathematics.
In particular, both elementary and special functions are either defined in terms of an infinite series or have series representations of one sort or another. For example, the Hurwitz zeta function is defined by the series
\begin{equation}
\zeta(z,q) := \sum_{n=0}^{\infty} \frac{1}{(n+q)^{z}}, 
\label{def-zeta Hurwitz}
\end{equation}
\noindent
for $z\in\allC, \realpart{z} > 1$ and $ q\in\allC, q \neq 0, \, -1, \, -2, \cdots$. Riemann zeta function is a special case of Hurwitz zeta function, $\zeta(z)=\zeta(z,1)$. 
\medskip

Conversely, confronted with an infinite series, it is always a legitimate pursuit to try to evaluate it in terms of known elementary or special functions. Sometimes, as an intermediate step, the evaluation of infinite series can be reduced to the evaluation of an integral, which may or may not have a closed form evaluation.
Consider, for instance, Plana's summation formula \cite{bateman v1},
\begin{equation}
\label{plana summation formula}
\sum\limits_{n = 0}^\infty  {f(n)}  = \frac{1}{2}f(0) 
+ \int_0^\infty  {f(x)dx}  
+ i\int_0^\infty  {\frac{{f(it) - f( - it)}}{{e^{2\pi t}  - 1}}dt} ,
\end{equation}
valid under certain restrictive growth conditions for the function $f$ in the complex domain.
\medskip

When applied to the series \eqref{def-zeta Hurwitz}, Plana's formula leads to Hermite's representation\cite{whittaker+watson},
\begin{equation}
\label{hermite1}
\zeta(z,q) = \frac{1}{2}q^{-z} + \frac{1}{z-1}q^{1-z} +
2 \int_{0}^{\infty} \frac{\sin\big(z \tan^{-1}(t/q)\big)}
{(q^2+t^{2})^{z/2} \, (e^{2 \pi t} -1)}\, dt,
\end{equation}
for $\realpart q>0$.
This representation actually provides a meromorphic extension of $\zeta(z,q)$ to the whole complex $z$-plane.
\medskip

A formula similar in spirit to Plana's summation formula can be obtained by use of contour integration and the residue theorem,
\begin{equation}
\label{sum by contour integration}
\sum\limits_{n = -\infty}^\infty  {f(n)}  = \frac{1}{2\pi i}
\mathop{\int\mkern-20.8mu \circlearrowleft}\nolimits_\updowncontour
{2\pi\nbe(z)f(-iz)dz}, 
\end{equation}
provided the function $f(z)$ does not have poles on the imaginary axis. The contour denoted above by $\updowncontour$ runs parallel to the imaginary axis, upwards from the right and downwards from the left, encircling counterclockwise all the poles (located at $z=ni, n\in\allZ$) of the kernel $\nbe(z)$, defined as
\begin{equation}
\label{def-BE-kernel}
\nbe(z) = \frac{1}{e^{2\pi z}-1} = \frac{1}{2}\left(\coth \pi z - 1\right).
\end{equation}
If the function $f(-iz)$ is such that its integral along a circular contour at infinity vanishes, then we can evaluate the contour integral in \eqref{sum by contour integration} by splitting the original contour $\updowncontour$ into two closed clockwise contours, one on each side of the imaginary axis and enclosing all the poles of $f(-iz)$. As a simple example, consider the evaluation of the sum
\begin{equation}
\label{simplest sum}
\sum\limits_{n = -\infty}^\infty  \frac{1}{n^2+q^2},
\end{equation}
where $q$ is a positive real variable. In this case we have the evaluation
\begin{align*}
\frac{1}{2\pi i}\mathop{\int\mkern-20.8mu \circlearrowleft}\nolimits_\updowncontour
{2\pi\nbe(z)\frac{1}{-z^2+q^2}dz}
&=
2\pi \sum_{z=\pm q} \underset{z}{\Res}\left[ \nbe(z)\frac{1}{z^2-q^2}\right]\cr
&=
\frac{2\pi}{2q}\left[ \nbe(q) - \nbe(-q) \right].
\end{align*}
The identity
\begin{equation}
\label{nbe-identity-1}
\nbe(z)+\nbe(-z)+1 = 0
\end{equation}
can be used to obtain the known result,
\begin{equation}
\sum\limits_{n = -\infty}^\infty  \frac{1}{n^2+q^2}=\frac{\pi\coth \pi q}{q}.
\end{equation}
\medskip

In this paper we consider the evaluation of a class of sums, $\MatsubaraSums$, that generalizes the simple sum \eqref{simplest sum}. Each of the sums in $\MatsubaraSums$ is defined in terms of a particular kind of connected graph, in a way that we make explicit in the next section.

In principle, the sums in $\MatsubaraSums$, called \emph{Matsubara sums}, can be evaluated by the repeated direct application of the contour integration formula \eqref{sum by contour integration}, on a case by case basis. However, using an algebraic identity, M.~Gaudin \cite{gaudin} has been able to obtain a closed form evaluation of any Matsubara sum as a sum of terms corresponding to the \emph{trees} of the corresponding graph. This is reviewed in section \ref{sec-explicit evaluations}.

Starting from Gaudin's result, we prove that any Matsubara sum can be alternatively evaluated by applying a linear operator to the evaluation of an integral associated with the sum. Although this integral can also be computed using Gaudin's method, it is usually the case that the direct computation of the integral can be done in a straightforward manner by other means, including symbolic manipulation programs such as Mathematica or Maple. In this case the evaluation of the corresponding sum will be notably simplified.


The general results presented in this paper are a by-product of the study of general properties of Feynman graphs in the so-called Euclidean or imaginary-time formalism of finite-temperature quantum field theory. In this formalism the evaluation of each Feynman graph requires the computation of a sum of the Matsubara type. The existence of the linear operator referred to above was first conjectured from the analysis of two general classes of Feynman graphs \cite{espinosa+stockmeyer-2004} and then established in full generality \cite{espinosa-2005}. The aim of this paper is to present the relevant mathematical results for a readership of non-(particle) physicists. Accordingly, all reference to physical quantities has been removed.
\medskip

Before defining our class $\MatsubaraSums$ in full generality and presenting the general theorems, we will illustrate our main results for the simplest non-trivial sum in $\MatsubaraSums$, which we will call $S_{G_2}$:

\begin{equation}
\label{def-sumG2}
S_{G_2}:=
\sum\limits_{n_1  =  - \infty }^\infty  {\sum\limits_{n_2  =  - \infty }^\infty  {\frac{{\delta_{n_1  + n_2  - N,0} }}{{\left( {n_1^2  + q_1^2 } \right)\left( {n_2^2  + q_2^2 } \right)}}} }=\sum\limits_{n =  - \infty }^\infty  {\frac{1}{{\left( {n^2  + q_1^2 } \right)\left( {(N - n)^2  + q_2^2 } \right)}}}.
\end{equation}
Here $N\in\allZ$ and $q_1>0$, $q_2>0$ are real variables. The Kronecker delta symbol is defined for $n,m\in\allZ$ by
\begin{equation}
\label{def-Kronecker delta}
\delta_{n,m}=
\begin{cases}
1 & \text{if $n=m$},\\
0 & \text{otherwise}. 
\end{cases}
\end{equation}

A standard evaluation, by the method of residues \eqref{sum by contour integration} for example, yields
\begin{multline}
\label{sumG2-result}
\sum\limits_{n =  - \infty }^\infty  {\frac{1}
{{\left( {n^2  + q_1^2 } \right)\left( {(N - n)^2  + q_2^2 } \right)}}}  = \frac{{2\pi }}
{{2q_1 2q_2 }}\left[ {\frac{{1 + \nbe(q_1 ) + \nbe(q_2 )}}
{{Ni + q_1  + q_2 }} } \right.
\\
\left. { + \frac{{\nbe(q_1 ) - \nbe(q_2 )}}
{{Ni - q_1  + q_2 }} - \frac{{\nbe(q_1 ) - \nbe(q_2 )}}
{{Ni + q_1  - q_2 }} - \frac{{1 + \nbe(q_1 ) + \nbe(q_2 )}}
{{Ni - q_1  - q_2 }}} \right],
\end{multline}
where $\nbe(q)$ is the kernel defined in \eqref{def-BE-kernel}.
\medskip

We note that, although written using the imaginary unit $i=\sqrt{-1}$, the result is clearly real. In the form just written, it becomes easier to recognize the general structure of the results we shall present below.
\medskip

Consider the integral obtained by replacing, in the LHS of \eqref{sumG2-result}, the sum over the discrete variable $n$ by an integral over a continuous variable $x$:
\begin{align}
\label{intG2-result}
I_{G_2}:=
\int_{ - \infty }^\infty  {\frac{{dx}}
{{\left( {x^2  + q_1^2 } \right)\left( {(N - x)^2  + q_2^2 } \right)}}}  &= \frac{{\pi (q_1  + q_2 )}}
{{q_1 q_2 \left( {N^2  + (q_1  + q_2 )^2 } \right)}}
\cr
&= \frac{{2\pi }}
{{2q_1 2q_2 }}\left[ {\frac{1}
{{Ni + q_1  + q_2 }} - \frac{1}
{{Ni - q_1  - q_2 }}} \right].
\end{align}

It is clear from these explicit evaluations, that for this simple case there is a simple formal relation between the sum $S_{G_2}$ and the integral $I_{G_2}$. To formulate this relation we introduce {\em reflection operators} $\opR_i \equiv \opR(q_i)$, acting on the space of functions of the variables $q_j$,
\begin{equation}
\label{def-reflection operator}
\opR(q_i)f(q_1,\ldots,q_i,\ldots)=f(q_1,\ldots,-q_i,\ldots).
\end{equation}
For instance, we have
\begin{align*}
\opR(q_1)\left[
 \frac{{2\pi }}{{2q_1 2q_2 }}\frac{{1}}{{Ni - q_1  - q_2 }}
 \right]
& = -\frac{{2\pi }}{{2q_1 2q_2 }}\frac{{1}}{{Ni + q_1  - q_2 }},
\intertext{and}
\opR(q_2)\left[
 \frac{{2\pi }}{{2q_1 2q_2 }}\frac{{1}}{{Ni + q_1  + q_2 }}
 \right]
& = -\frac{{2\pi }}{{2q_1 2q_2 }}\frac{{1}}{{Ni + q_1  - q_2 }}.
\end{align*}
\medskip

In terms of the reflection operator \eqref{def-reflection operator} we have:
\begin{equation}
\label{torG2}
S_{G_2}(N,q_1,q_2)=\left[ 1 + \nbe(q_1)\big( 1-\opR(q_1) \big)
+ \nbe(q_2)\big( 1-\opR(q_2) \big) \right] I_{G_2}(N,q_1,q_2).
\end{equation}

For instance, the term
\begin{equation*}
 - \frac{{2\pi }}{{2q_1 2q_2 }}\frac{{\nbe(q_1 )}}{{Ni + q_1  - q_2 }}
\end{equation*}
in \eqref{sumG2-result} is generated as
\begin{equation*}
 - \frac{{2\pi }}{{2q_1 2q_2 }}\frac{{\nbe(q_1 )}}{{Ni + q_1  - q_2 }}
=
-\nbe(q_1)\opR(q_1)\left[
 - \frac{{2\pi }}{{2q_1 2q_2 }}\frac{{1}}{{Ni - q_1  - q_2 }}
\right].
\end{equation*}

It is straightforward to show that the operator $\big(1 - \opR_1\big)\big(1 - \opR_2\big)$ annihilates the integral $I_{G_2}$ evaluated in \eqref{intG2-result}, that is:
\begin{equation}
\big(1 - \opR_1\big)\big(1 - \opR_2\big) I_{G_2} \equiv 0,
\end{equation}
where we have abbreviated $\opR_i := \opR(q_i)$.
Therefore, the operator in \eqref{torG2} that generates the sum $S_{G_2}$ from the corresponding integral $I_{G_2}$ can be written in the multiplicative form,
\begin{equation}
\label{TOp G2}
\TOp_{G_2} =\prod_{i=1}^2 \left[ 1 + \nbe_i\big(1 - \opR_i\big) \right],
\end{equation}
\medskip
where $\nbe_i:=\nbe(q_i)$.

In this paper we shall prove that all the sums in the class $\MatsubaraSums$, to be defined in the next section, satisfy a property similar to \eqref{torG2}, with an operator of the type \eqref{TOp G2}.

\begin{Note}
It is important to notice that whereas the evaluation \eqref{intG2-result}
is still valid if the variable $N$ is extended to the real or complex domains, the same does not happen in the case of \eqref{sumG2-result}. In the latter case one finds, for $N\in\allC$,
\begin{multline}
\label{sumG2complex-result}
\sum\limits_{n =  - \infty }^\infty  {\frac{1}
{{\left( {n^2  + q_1^2 } \right)\left( {(N - n)^2  + q_2^2 } \right)}}}  = \frac{{2\pi }}
{{2q_1 2q_2 }}\left[ {\frac{{1 + \nbe(q_1 ) + \nbe(q_2 + Ni)}}
{{Ni + q_1  + q_2 }} } \right.
\\
\left. { + \frac{{\nbe(q_1 ) - \nbe(q_2 + Ni )}}
{{Ni - q_1  + q_2 }} - \frac{{\nbe(q_1 ) - \nbe(q_2 - Ni)}}
{{Ni + q_1  - q_2 }} - \frac{{1 + \nbe(q_1 ) + \nbe(q_2 - Ni)}}
{{Ni - q_1  - q_2 }}} \right].
\end{multline}
Therefore it is clear that the simple relationship \eqref{torG2} only holds in the case $N\in\allZ$. The Matsubara sums to be defined next will in general be functions of several integer variables $N_v$. We will not consider in this paper the extension of our results to the case $N_v$ complex.
\end{Note}

\section{Matsubara Sums}
\label{Matsubara Sums}

In this section we will introduce the class of sums $\MatsubaraSums$, whose elements we shall call \emph{Matsubara Sums}. These sums were considered for the first time by T.~Matsubara \cite{matsubara} in his work on the statistical mechanics of quantum fields, where they appear in connection to the evaluation of so-called Feynman diagrams. 
In order to give a general definition of the class $\MatsubaraSums$ we need to make use of some graph-theoretical terminology. 

Consider a connected graph formed by a set of $\NV$ points (also called vertices) and $\NI$ edges (also called lines or arcs). We will demand that each line joins two different vertices (that is, we exclude loops, \emph{i.e.} lines that join a vertex to itself) and that each vertex be the endpoint of at least two different lines. Any pair of vertices can be joined by more than one line.
In graph-theoretic language \cite{berge}\cite{harary}, we are considering a $(\NV,\NI)$-multigraph such that the degree of each vertex is at least 2.

We will restrict our attention to graphs of the type described, which we shall call \emph{Matsubara graphs}.
\medskip

Let $G$ be a Matsubara graph with $\NV$ vertices and $\NI$ lines. 
We choose for each line a definite orientation and assign to this oriented line a positive real number $q_i$ and an integer-valued summation variable $n_i$. We assign to each vertex $v$ an integer $N_v$ and the algebraic sum $T_v:=\sum_{i} s^v_i n_i$,
where
\begin{equation}
s^v_i=
\begin{cases}
+1 & \text{if the line $i$ is oriented \emph{into} vertex $v$} \\
-1 & \text{if the line $i$ is oriented \emph{away from} vertex $v$} \\
0 & \text{if the line $i$ is not incident on vertex $v$}.
\end{cases}
\end{equation}


The \emph{Matsubara sum} $S_G$ associated to the graph $G$ is defined as
\begin{equation}
\label{def-SG}
S_G:=
\sum\limits_{n_1  =  - \infty }^\infty
\sum\limits_{n_2  =  - \infty }^\infty
\cdots \sum\limits_{n_I  =  - \infty }^\infty  
{\frac{{\delta_G (n_1 ,n_2 , \ldots ,n_I;\{N_v\} )}}
{{\left( {n_1^2  + q_1^2 } \right)\left( {n_2^2  + q_2^2 } \right) \cdots \left( {n_I^2  + q_I^2 } \right)}}}  .
\end{equation}
Here $\delta_G(n_1 ,n_2 , \ldots ,n_I;\{N_v\}  )$ is a function valued in $\{0,1\}$ whose function is to impose a series of linear constraints among the summation variables $n_i$.
It is given explicitly by
\begin{equation}
\delta_G(n_1 ,n_2 , \ldots ,n_I;\{N_v\}  )=\prod_{v=1}^{V}\delta_{T_v,N_v},
\end{equation}
where $\delta_{m,n}$ is Kronecker delta defined in \eqref{def-Kronecker delta} above.
Whenever there is no possibility of confusion we will use the shorthand notation $\delta_G(n;N)$ to denote the object $\delta_G(n_1 ,n_2 , \ldots ,n_I;\{N_v\}  )$.

\medskip
\begin{lemma}
$S_G=0$, unless the integers $N_v$ satisfy the relation
\begin{equation}
\label{cond-sum of N's vanishes}
\sum_{v=1}^{\NV} N_v = 0.
\end{equation}
\end{lemma}
\begin{proof}
The $\NV$ equations $T_v = N_v$ have to be satisfied simultaneously for the sum $S_G$ not to vanish. Now, each summation variable $n_i$ (associated to line $i$) appears in only two of these equations (those corresponding to the vertices on which line $i$ is incident), once in the form $+n_i$ and once in the form $-n_i$. Therefore, $\sum_v T_v \equiv 0$ and
\begin{equation*}
\sum_v N_v = \sum_v T_v = 0. 
\end{equation*}
\end{proof}

Hence, in order for the Matsubara sums to be considered not to vanish identically, we shall always assume that the condition \eqref{cond-sum of N's vanishes} holds.
\medskip

We now give a few examples of Matsubara sums.

\begin{example}
The simplest non-trivial example of a Matsubara sum corresponds to the $(2,2)$-graph $G_2$ containing 2 vertices joined together by two lines, with $N_1+N_2\equiv 0$, shown in figure \ref{fig-G2}. Setting $N_1=N$, the Matsubara sum for $G_2$ is simply the sum $S_{G_2}$ defined in \eqref{def-sumG2} in the Introduction.
\end{example}
\begin{figure}
\centerline{\epsfig{file=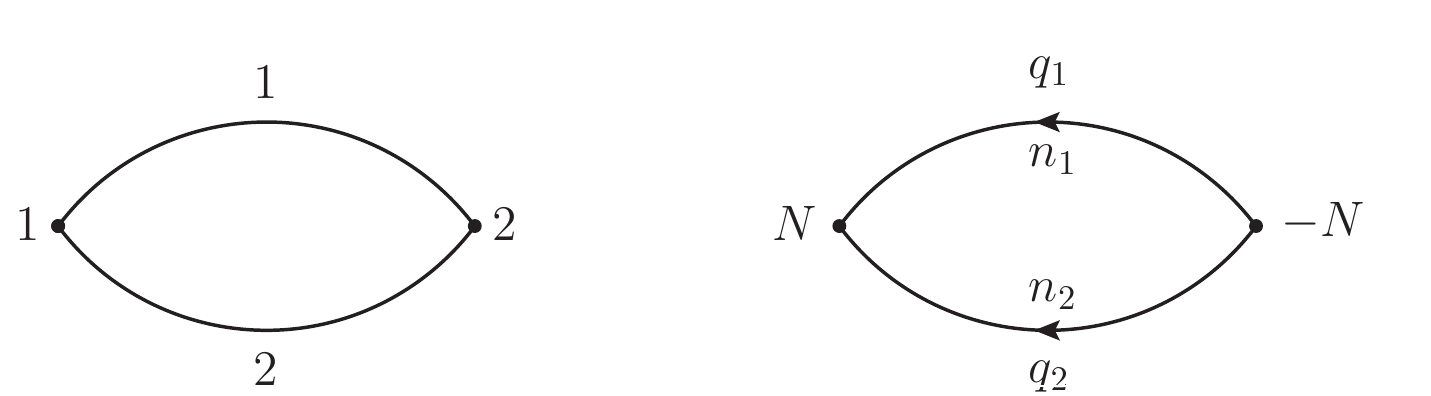,height=2.8cm,angle=0}}
\caption[]{The graph $G_2$, with labeled vertices and lines (left) and its Matsubara sum variables (right).}
\label{fig-G2}
\end{figure}

\medskip

\begin{example}
A slightly more complicated example of a Matsubara sum is the one associated to the $(2,3)$-graph $G_3$, consisting of two vertices joined now by three lines, represented in figure \ref{fig-G3}.
\begin{figure}
\centerline{\epsfig{file=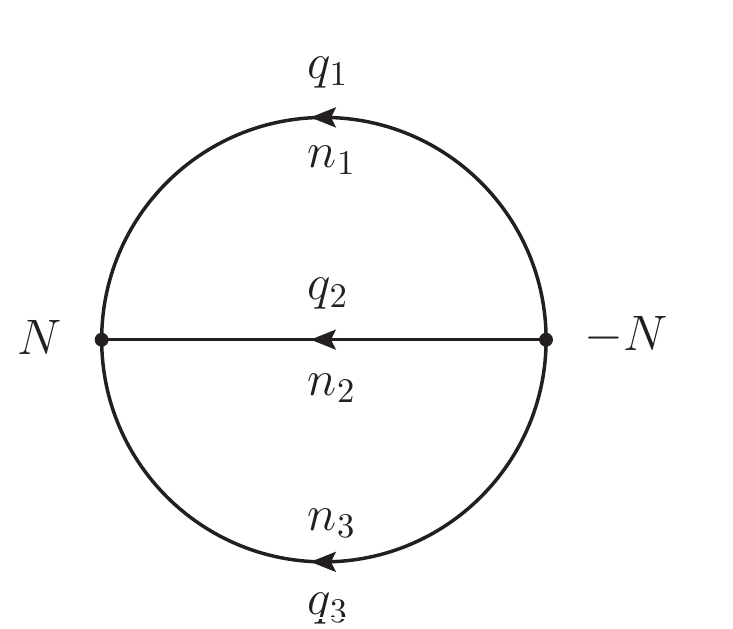,height=4.5cm,angle=0}}
\caption[]{The graph $G_3$ with its Matsubara sum variables.}
\label{fig-G3}
\end{figure}
Again we require $N_2 = -N_1 \equiv -N$, in which case the constraints at the vertices reduce to the single equation $n_1+n_2+n_3=N$. Thus,
\begin{equation}
S_{G_3}(N,q_1,q_2,q_3)=
\sum\limits_{n_1 } {\sum\limits_{n_2 } {\frac{1}
{{\left( {n_1^2  + q_1^2 } \right)\left( {n_2^2  + q_2^2 } \right)\left( {(N - n_1  - n_2 )^2  + q_3^2 } \right)}}} } ,
\end{equation}
where it is from now on understood that the summation variables $n_i$ run from $-\infty$ to $+\infty$.

As it was the case for the sum $S_{G_2}$, it turns out that $S_{G_3}$ can also be generated in a simple way from the corresponding integral $I_{G_3}$, obtained replacing the double sum by a double integral,
\begin{align}
I_{G_3}(N,q_1,q_2,q_3) &=
\int_{-\infty }^{\infty} dx_1 \int_{-\infty }^{\infty} dx_2 {\frac{1}
{{\left( {x_1^2  + q_1^2 } \right)\left( {x_2^2  + q_2^2 } \right)\left( {(N - x_1  - x_2 )^2  + q_3^2 } \right)}} }
\notag\\
&= \frac{{(2\pi)^2 }}
{{2q_1 2q_2 2q_3}}\left[ {\frac{1}
{{Ni + q_1  + q_2 + q_3}} - \frac{1}
{{Ni - q_1  - q_2 - q_3}}} \right].
\end{align}
In this case, after a lengthy evaluation, we find,
\begin{equation}
S_{G_3}(N,q_1,q_2,q_3)=\TOpp_{G_3}(q_1,q_2,q_3) I_{G_3}(N,q_1,q_2,q_3),
\end{equation}
where $\TOpp_{G_3}$ is the operator,
\begin{align}
\label{eq-opG3}
\TOpp_{G_3} = 1 &+\nbe(q_1)\big( 1-\opR(q_1) \big)
+ \nbe(q_2)\big( 1-\opR(q_2) \big) + \nbe(q_3)\big( 1-\opR(q_3) \big)
\cr
&+ \nbe(q_1)\nbe(q_2)\big( 1-\opR(q_1) \big)\big( 1-\opR(q_2) \big) 
\cr
&+ \nbe(q_1)\nbe(q_3)\big( 1-\opR(q_1) \big)\big( 1-\opR(q_3) \big) 
\cr
&+ \nbe(q_2)\nbe(q_3)\big( 1-\opR(q_2) \big)\big( 1-\opR(q_3) \big).
\end{align}

\bigskip

As in the case of $G_2$, it is direct to show that the operator $\big(1 - \opR_1\big)\big(1 - \opR_2\big)\big(1 - \opR_3\big)$ annihilates $I_{G_3}$:
\begin{equation}
\big(1 - \opR_1\big)\big(1 - \opR_2\big)\big(1 - \opR_3\big) I_{G_3} \equiv 0.
\end{equation}
Therefore, the operator that generates the sum $S_{G_3}$ from the corresponding integral $I_{G_3}$ can be written in the multiplicative form,
\begin{equation}
\TOp_{G_3} =\prod_{i=1}^3 \left[ 1 + \nbe_i\big(1 - \opR_i\big) \right].
\end{equation}
\medskip
\end{example}

We will prove in the following sections that the types of relationship just described are generic for Matsubara sums.

\begin{figure}
\centerline{\epsfig{file=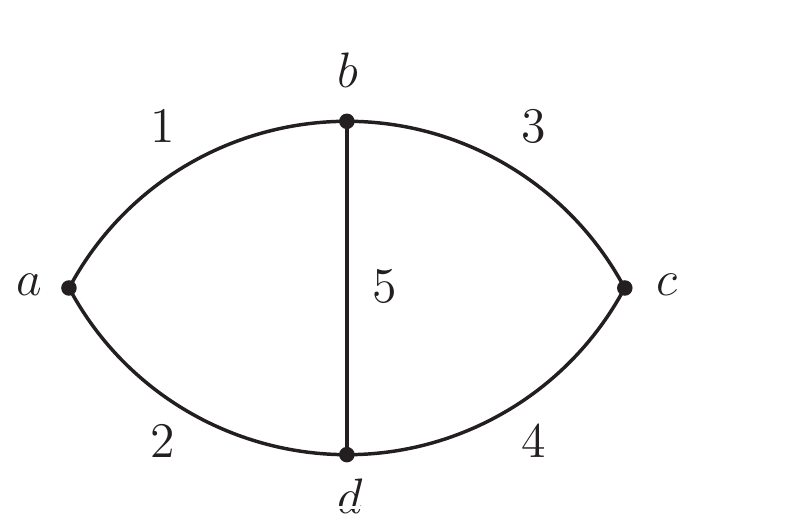,height=3.7cm,angle=0}}
\caption[]{The graph $G_4$, with labeled vertices and lines. The orientation of the lines, not shown, is described in the text.}
\label{fig-G4}
\end{figure}
\begin{example}
As a final, more intricate example we consider the Matsubara sum associated to the graph $G_4$, represented in figure \ref{fig-G4}. $G_4$ is a $(4,5)$-graph. Its Matsubara sum will be a function of 3 integer variables, $N_a, N_b, N_c$, and 5 real positive variables, $q_1,\ldots,q_5$:
\begin{equation}
\label{def-sumG4}
S_{G_4}:=
\sum\limits_{n_1 , \ldots ,n_5 } {\frac{{\delta_{n_1  + n_2 ,N_a } \delta_{n_3  - n_1  - n_5 ,N_b } \delta_{ - n_3  - n_4 ,N_c } }}
{{(n_1^2  + q_1^2 )(n_2^2  + q_2^2 )(n_3^2  + q_3^2 )(n_4^2  + q_4^2 )(n_5^2  + q_5^2 )}}}, 
\end{equation}
The orientation of the lines has been chosen such that line 5 flows from top to bottom and all the rest flow from right to left. Solving for the constraints imposed by the Kronecker deltas, we find that $S_{G_4}$ is equivalent to the double sum
\begin{multline}
\label{eq-sumG4}
\sum_{n=-\infty}^{\infty}\sum_{m=-\infty}^{\infty}
\left\{
{\frac{1}{{(n^2  + q_1^2 )(m^2  + q_5^2 )}}} \right.\\
\times
\left.
\frac{1}{((N_a  - n)^2  + q_2^2 )((N_b  + n + m)^2  + q_3^2 )((N_b  + N_c  + n + m)^2  + q_4^2 )} \right\}.
\end{multline}
We shall present the explicit evaluation of this sum in section \ref{sec-calculation of S4}.
\end{example}

\section{The main theorems}
\label{sec-main-theorems}

Here we state the main theorems concerning the explicit evaluation of Matsubara sums.
These results were first conjectured \cite{espinosa+stockmeyer-2004} and then proved \cite{espinosa-2005} in the context of thermal quantum field theory. Our proofs of theorems \ref{thm-main-1} and \ref{thm-main-2} and of lemma \ref{lemma:cutsets} rely on the explicit evaluation of both the Matsubara sum and the Matsubara integral associated to a general graph $G$, which will be given in the next section.

\begin{theorem}
\label{thm-main-1}
Let $G$ be a Matsubara graph. Let $\{q_i\}, i=1,\ldots,I$ be the set of positive real values associated to the lines of $G$ and $\{N_v\}, v=1,\ldots,V$ be the set of integer values associated to the vertices of $G$. We assume the condition $\sum_v N_v = 0$ to hold. Then the Matsubara sum of $G$, $S_G(\{N_v\},\{q_i\})$, can be explicitly evaluated as
\begin{equation}
S_G(\{N_v\},\{q_i\}) = \TOp_G(\{q_i\}) I_G(\{N_v\},\{q_i\}),
\end{equation}
where
\begin{equation}
\TOp_G(\{q_i\}) = \prod_{i=1}^I \left[ 1 + \nbe_i\big(1 - \opR_i\big) \right],
\end{equation}
and $I_G(\{N_v\},\{q_i\})$ is the Matsubara integral of $G$.
\end{theorem}
Upon expansion of the product, the operator $\TOp_G$ can be seen to contain one or more terms that individually annihilate the Matsubara integral $I_G$, so that actually the operator $\TOp_G$ can be made ``smaller''. To state this result we need the following graph-theoretical definition \cite{harary}:
\begin{definition}
A {\em cutset} of the (connected) graph $G$ is a set of lines whose removal from the graph results in a disconnected graph.
\end{definition}
\begin{lemma}
\label{lemma:cutsets}
Let $C$ be a cutset of the graph $G$. Then the operator
\begin{equation}
\hat{\mathcal{A}}_{C}=\prod_{i\in C}\big(1 - \opR_i\big)
\end{equation}
annihilates the Matsubara integral $I_G$ associated to the graph $G$.
\end{lemma}

\begin{theorem}
\label{thm-main-2}
Let the conditions of theorem \ref{thm-main-1} hold and let $\NL:=\NI-\NV+1$ be the number of independent cycles of the graph $G$. Then the relationship between the Matsubara sum $S_G(\{N_v\},\{q_i\})$ and integral $I_G(\{N_v\},\{q_i\})$ can alternatively be written as
\begin{equation}
S_G(\{N_v\},\{q_i\}) = \TOpp_G(\{q_i\}) I_G(\{N_v\},\{q_i\}),
\end{equation}
where
\begin{equation}
\label{eq:thermal-operator-form2}
\begin{split}
\TOpp_G(\{q_i\}):=1+&\sum_{i=1}^\NI  \nbe_i\big(1-\opR_i\big)
+\sideset{}{^\prime}\sum_{\langle i_1,  i_2\rangle} \nbe_{i_1} \nbe_{i_2}
\big(1-\opR_{i_1}\big)\big(1-\opR_{i_2}\big)\\
+& \dots + \sideset{}{^\prime}\sum_{\langle i_1, \dots,  i_{\NL}\rangle}
\prod_{l=1}^{\NL} \nbe_{i_l}\big(1-\opR_{i_l}\big).
\end{split}
\end{equation}
Here the indices $i_1, i_2, \ldots$ run
from 1 to $\NI$ (the number of lines of the graph $G$) and the symbol
$\langle i_1, \dots,  i_k\rangle$ stands for an unordered
$k$-tuple with no repeated indices, representing a particular set
of lines. The prime on the summation symbols imply that
the tuples that are cutsets of the graph $G$ are to be excluded from the sums.
\end{theorem}

Note that the operator $\TOpp_G(\{q_i\})$ contains products of at most
$\NL$ kernel factors $\nbe(q_i)$, since for a graph with $\NL$ independent cycles
the maximum number of lines that can be removed without
disconnecting the graph is precisely $\NL$.
\footnote{The number of independent cycles of a graph $G$ is called the \emph{cycle rank} or \emph{cyclomatic number} in graph theory and is given by $m(G)=\NI-\NV+1$ if $G$ is connected \cite{harary}.}

\section{Explicit evaluations of $S_G$ and $I_G$}
\label{sec-explicit evaluations}
We notice that the main building block of a Matsubara sum can be expressed as
\begin{align}
\frac{1}{{n^2  + q^2 }} &= \frac{1}{{2q}}
\left[ {\frac{1}{{in + q}} - \frac{1}{{in - q}}} \right]\cr
&=\frac{1}{{2q}}\big(1-\opR(q)\big)\frac{1}{{q - in}}.
\end{align}
The representation above would allow us trade the original quadratic denominators in the Matsubara sum \eqref{def-SG} for linear denominators (in the summation variables $n_i$), and express $S_G$ in the form
\begin{equation}
\label{SG-gaudin}
S_G=\prod\limits_{k = 1}^\NI {\frac{1}
{{2q_k }}\big(1 - \opR_k \big)} \sum\limits_{n_1 , \ldots ,n_\NI } {\frac{{\delta_G (n_1 ,n_2, \ldots ,n_\NI; \{N_v\} )}}
{{(q_1 - in_1 )(q_2 - in_2 ) \cdots ( q_\NI - in_\NI  )}}},
\end{equation}
were it not for the fact the new sum does not converge in general. However, as we shall see, it is possible to regulate the sum in \eqref{SG-gaudin} in such a way that it is well defined.
\medskip

The main results that will allow us to obtain a complete evaluation of the Matsubara sum $S_G$ were obtained by M.~Gaudin \cite{gaudin} long ago, and will be reviewed in this section, adapted to the context of this paper.
Gaudin showed that the summand in \eqref{SG-gaudin} admits a decomposition into partial fractions, which allows us to systematically eliminate the constrains imposed by the delta function $\delta_G(n;N)$ and perform the sum explicitly.

The generalized Kronecker
delta $\delta_G(n;N)$ enforces $\NV-1$ independent linear relations
satisfied by the summation variables $n_i$, also involving the
vertex parameters $N_v$, which we shall write as
\begin{align}
\label{R:vertex constraints}
R_v (N,n) = 0,\quad\text{for }v = 1, \ldots ,\NV - 1.
\end{align}

This system of linear equations allows us to solve for $\NV - 1$
of the $\NI$ summation variables in terms of a set of $\NL=\NI-\NV+1$
independent ones. In general, there will be several distinct ways of
choosing this set of independent summation variables.
As shown by Gaudin\cite{gaudin}, there is a one-to-one correspondence between
the collection of all possible sets of independent summation variables and
the set of all \emph{trees} associated to the given (connected) graph
$G$.

A tree\footnote{More accurately, a \emph{spanning} tree, in the graph-theoretical terminology.}
is a set of lines of $G$ joining all vertices and making a connected
graph with no cycles. Every tree $\T$ will contain $\NV-1$ lines and its
complement $\Tbar$ (the set of lines of $G$ which do not belong to $\T$) will
have $\NL$ lines. The summation variables corresponding to the $\NL$ lines in
$\Tbar$, denoted by $n_l$, will constitute a set of independent summation variables in terms of which the system \eqref{R:vertex constraints} can be solved.
The summation variables
associated with the lines of the tree, $n_j$, with $j\in\T$,
will be linear combinations of the independent summation variables
and the vertex parameters,
\begin{equation}
n_j=\Omega_j^{\T}(N,n_l),\quad j\in{\T}, l\in{\Tbar}.
\end{equation}
\begin{figure}
\centerline{\epsfig{file=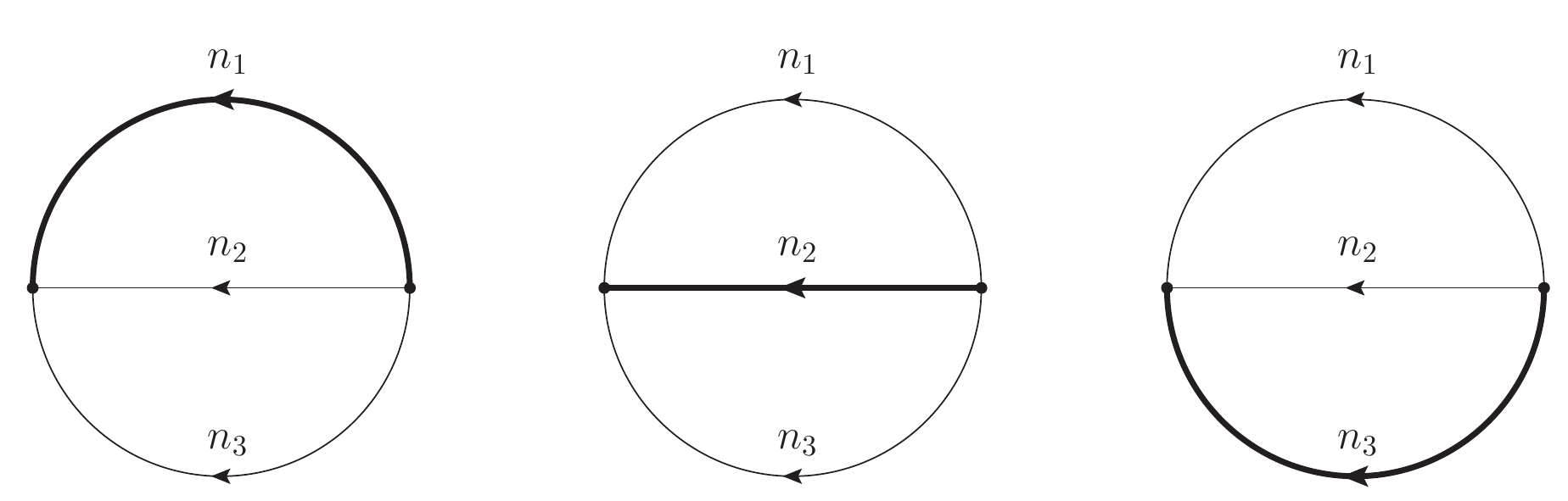,height=3.8cm,angle=0}}
\caption[]{The trees of the graph $G_3$, shown as dark lines.}
\label{fig-G3trees}
\end{figure}

As a simple example, in figure \ref{fig-G3trees} we show the three
possible trees for the $(2,3)$-graph $G_3$.
In this case, each tree $\T$ is composed by a single line
(heavy line), whose summation variable can be expressed, after solving for the constraint at one of the vertices imposed by the delta function, in terms
of the two independent summation variables associated with the
(thin) lines that do not belong to the tree (these are the lines in $\Tbar$) and the vertex parameter $N$. For instance, for the first tree we have
$n_1 = N - n_2 - n_3$, etc.

\begin{figure}
\centerline{\epsfig{file=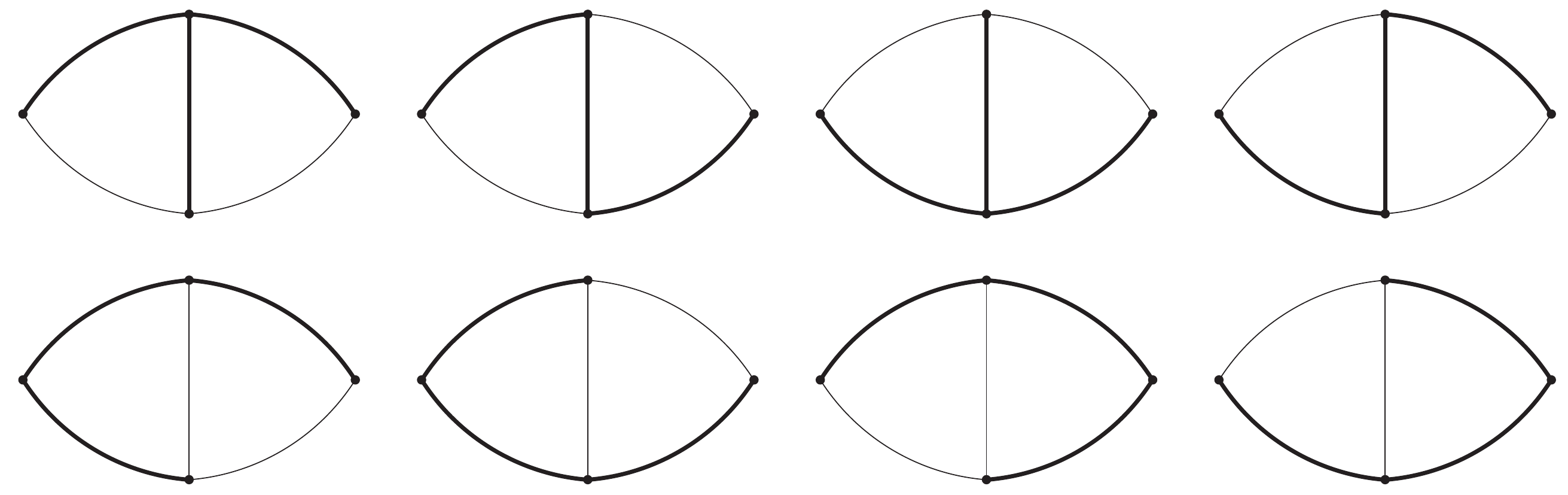,width=12cm,angle=0}}
\caption[]{The 8 trees of the graph $G_4$, represented by the dark lines.}
\label{fig-G4trees}
\end{figure}

Similarly, for the $(4,5)$-graph $G_4$ of figure \ref{fig-G4} there exist 8 trees, shown in figure \ref{fig-G4trees}. The system of equations determined by the Kronecker deltas in the definition \eqref{def-sumG4} of $S_{G_4}$ can be solved in terms of 8 different sets of independent variables, one for each of the 8 trees. For instance, the solution that led to the double sum \eqref{eq-sumG4}, $n_2=N_a-n_1$, $n_3=N_b+n_1+n_5$ and $n_4=-N_b-N_c-n_1-n_5$, corresponds to the last tree in figure \ref{fig-G4trees} (the one in the bottom-right corner).
\medskip

Gaudin's main insight is the following identity for the rational function
appearing in \eqref{SG-gaudin}:
\begin{equation}
\label{gaudin-identity}
\prod\limits_{k = 1}^\NI {\frac{1}{{q_k - in_k }}}\delta_G(n;N)
= \sum\limits_{\T} {\prod\limits_{j \in \T} {\frac{1}
{{q_j  - i\Omega _j^{\T} (N, -iq_l )}}} }
\prod\limits_{l \in\Tbar} {\frac{1}{{q_l - in_l}}}\delta_G(n;N).
\end{equation}

For instance, for the case of the Matsubara sum $S_{G_3}$ this identity takes the
form
\begin{multline}
\label{gaudin-identity-G2}
{\frac{{\delta \left( {N - n_1  - n_2  - n_3 } \right)}}
{{\left( {q_1  - in_1 } \right)\left( {q_2  - in_2 } \right)\left( {q_3  - in_3 } \right)}}}  = \frac{{\delta \left( {N - n_1  - n_2  - n_3 } \right)}}
{{\left( {q_1  + q_2  + q_3  - iN} \right)}}
\\
\times \left\{ { {\frac{1}
{{\left( {q_2  - in_2 } \right)\left( {q_3  - in_3 } \right)}}}  + {\frac{1}
{{\left( {q_1  - in_1 } \right)\left( {q_3  - in_3 } \right)}}}  + {\frac{1}
{{\left( {q_1  - in_1 } \right)\left( {q_2  - in_2 } \right)}}} }\right\}.
\end{multline}
We note that the identity \eqref{gaudin-identity} holds only in the presence of the constraint $\delta_G(n;N)$ on the variables $n_i$. For instance, in the example above, 
\begin{multline*}
{ {\frac{1}
{{\left( {q_2  - in_2 } \right)\left( {q_3  - in_3 } \right)}}}  + {\frac{1}
{{\left( {q_1  - in_1 } \right)\left( {q_3  - in_3 } \right)}}}  + {\frac{1}
{{\left( {q_1  - in_1 } \right)\left( {q_2  - in_2 } \right)}}} } =
\\
\frac{q_1 + q_2 + q_3 - i(n_1 + n_2 + n_3)}{( {q_1  - in_1 } )( {q_2  - in_2 } )( {q_3  - in_3 } )},
\end{multline*}
so that the identity \eqref{gaudin-identity-G2} holds only if the constraint $n_1 + n_2 + n_3 = N$ is imposed.
\medskip

Once we have applied Gaudin's identity \eqref{gaudin-identity} to the Matsubara sum \eqref{SG-gaudin}, we proceed to perform the sum over the variables $n_i$. Since, for each tree, the summation variables $n_j,j\in\T$, appear now only in the constraint $\delta_G(n;N)$, we have formally
\begin{equation}
\sum\limits_{n_1 , \ldots ,n_I } {\prod\limits_{l \in \Tbar} {\frac{1}
{{q_l  - in_l }}} \delta_G (n;N)}  = \sum\limits_{ n_l : l \in \Tbar} {\prod\limits_{l \in \Tbar} {\frac{1}{{q_l  - in_l }}} } =
{\prod\limits_{l \in \Tbar}}\sum\limits_{n_l} {\frac{1}{{q_l  - in_l }}}.
\end{equation}
Unfortunately, the sum
\begin{equation}
\sum_{n=-\infty}^{\infty} {\frac{1}{{q  - in }}}
\end{equation}
diverges, and the result above does not make sense.
\\
However, as shown in \cite{gaudin}, it is possible to regulate the sum in such a way that all intermediate steps are mathematically sound. The basic idea is to associate to each line $i$ in the graph $G$ a real parameter $\tau_i$, which at the end will be taken to zero, and consider the sum
\begin{equation}
\label{def-SG-regulated}
S_G^\tau:=
\sum\limits_{n_1 , \ldots ,n_\NI }
\frac{e^{i(n_1\tau_1+n_2\tau_2+\cdots+n_\NI\tau_\NI)}}
{{\left( {n_1^2  + q_1^2 } \right)\left( {n_2^2  + q_2^2 } \right) \cdots \left( {n_\NI^2  + q_I^2 } \right)}}{\delta_G(n_1 ,n_2 , \ldots ,n_I;\{N_v\}  )},
\end{equation}
so that
\begin{equation}
S_G = \lim_{\tau_i\to 0} S_G^\tau.
\end{equation}
The subtle issues related to the interchange in the order of limits and summations are discussed at length in \cite{gaudin}. Now we have
\begin{equation}
\sum\limits_{n_1 , \ldots ,n_I } {\prod\limits_{l \in \Tbar} {\frac{1}
{{q_l  - in_l }}} \delta_G (n;N)} e^{i(n_1\tau_1+n_2\tau_2+\cdots+n_\NI\tau_\NI)} = 
{\prod\limits_{l \in \Tbar}}\sum\limits_{n_l} {\frac{e^{in_l T_l}}{{q_l  - in_l }}}.
\end{equation}
where $T_l$ is a linear combination of the $\tau_i$, whose particular form depends on the tree $\T$ being considered. It can be shown that $T_l$ is the algebraic sum of the $\tau$-variables of the lines of the cycle $\T\cup\{l\}$ (formed by adding the line $l$ to the tree $\T$): $\tau_k$ will be preceded by a plus sign if the line $k$ has the same orientation as the line $l$, and by a minus sign otherwise.
\\
The sum that we now need is given by
\begin{equation}
\sum\limits_{n} {\frac{e^{in T}}{{q  - in }}} = 2\pi\eps\nbe(\eps q)e^{T q},
\end{equation}
where $\eps$ is the sign of $T$ ($\eps=+1$ if $T>0$ and $\eps=-1$ if $T<0$)
and $\nbe(z)$ is the kernel introduced in \eqref{def-BE-kernel}. Clearly, only $\eps$ matters when the regulator $T$ is taken to zero.
\medskip

The final result for the Matsubara for the graph $G$ is therefore
\begin{equation}
S_G=\prod\limits_{k = 1}^\NI {\frac{1}{{2q_k }}\big(1 - \opR_k \big)}
\sum\limits_{\T} \left(
{\prod\limits_{j \in \T} {\frac{1}{{q_j  - i\Omega _j^{\T} (N, -iq_l )}}}}
\prod\limits_{l \in\Tbar} 2\pi\eps_l \nbe(\eps_l q_l) \right),
\end{equation}
where $\eps_l$ is the sign of the variable $T_l$ associated to each line $l\in\Tbar$ in the process of regulation.
\\
\\
Consider now the (regulated) Matsubara integral associated to the graph $G$:
\begin{multline}
\label{def-IG-regulated}
I_G^\tau:=
\idotsint dx_1 dx_2 \cdots dx_\NI \\
\frac{e^{i(x_1\tau_1+x_2\tau_2+\cdots+x_\NI\tau_\NI)}}
{{\left( {x_1^2  + q_1^2 } \right)\left( {x_2^2  + q_2^2 } \right) \cdots \left( {x_\NI^2  + q_I^2 } \right)}}{\delta_G (x_1 ,x_2 , \ldots ,x_\NI; \{N_v\} )},
\end{multline}
so that
\begin{equation}
I_G = \lim_{\tau_i\to 0} I_G^\tau.
\end{equation}
Now $\delta_G$ stands for a product of Dirac delta functions, and the integrals over the $x$-variables run from $-\infty$ to $\infty$.
\\
All the algebraic manipulations given above for the sum $S_G$ hold here unchanged, leading us to
\begin{multline}
\idotsint dx_1 dx_2 \cdots dx_\NI
{\prod\limits_{l \in \Tbar} {\frac{1} {{q_l  - ix_l }}} \delta_G (x,N)} e^{i(x_1\tau_1+x_2\tau_2+\cdots+x_\NI\tau_\NI)} \\
= 
{\prod\limits_{l \in \Tbar}}
\int_{-\infty}^{\infty} {\frac{e^{ix_l T_l}}{{q_l  - ix_l }}} dx_l,
\end{multline}
where $T_l$ is the same linear combination of the $\tau_i$ as in the case of $S_G$.
\\
The integral that we need now is given by
\begin{equation}
\int_{-\infty}^{\infty} {\frac{e^{ix T}}{{q  - ix }}} dx = -2\pi\eps\Heaviside(-\eps q)e^{T q},
\end{equation}
where $\eps$ is the sign of $T$, and $\Heaviside(q)$ is the Heaviside step function. Clearly, again only $\eps$ matters when the regulator $T$ is taken to zero.
\\
So,
\begin{equation}
\label{IG-explicit}
I_G=\prod\limits_{k = 1}^\NI {\frac{1}{{2q_k }}\big(1 - \opR_k \big)}
\sum\limits_{\T} \left(
{\prod\limits_{j \in \T} {\frac{1}{{q_j  - i\Omega _j^{\T} (N, -iq_l )}}}}
\prod\limits_{l \in\Tbar} -2\pi\eps_l \Heaviside(-\eps_l q_l) \right).
\end{equation}

\section{The relation between $S_G$ and $I_G$}

In this section we use the explicit evaluations obtained in the previous section to prove the theorems stated in section \ref{sec-main-theorems}, which provide an efficient way of computing the Matsubara sum $S_G$ in terms of the Matsubara integral $I_G$.

\begin{Lem}
\label{lemma-1}
The Matsubara integral can be written as
\begin{equation}
\label{eq:IG-sum over trees}
I_G=\sum\limits_{\T} I_G^{\T},
\end{equation}
where
\begin{equation}
\label{IG-tree-explicit}
I_G^{\T} = 
(2\pi)^\NL
\prod\limits_{k = 1}^\NI {\frac{1}{{2q_k }}}
{\prod\limits_{j \in \T} \big(1 - \opR_j \big)
{\frac{1}{{q_j  - i\Omega _j^{\T} (N, i\eps_l q_l )}}}}
\end{equation}
is the contribution to $I_G$ associated to the tree $\T$.
\end{Lem}
\begin{proof}
Assuming $q_l>0$ we have
\begin{align*}
\big(1 - \opR_l \big)\left[\Heaviside(-\eps_l q_l)f(q_l)\right]&=\Heaviside(-\eps_l q_l)f(q_l)-
\Heaviside(\eps_l q_l)f(-q_l)\cr
&=-\eps_l f(-\eps_l q_l).
\end{align*}
Therefore, for each tree $\T$, the action of the operators $\big(1 - \opR_l \big)$ with $l\in\Tbar$ can be performed explicitly in \eqref{IG-explicit}, to yield \eqref{IG-tree-explicit}. Note that $\eps_l^2=1$ and that there are $\NL=\NI-\NV+1$ factors $(2\pi)$ (one for each line en $\Tbar$). 
\end{proof}

\begin{Lem}
\label{lemma-2}
For $q\in\allR$, $q\neq 0$, the function $\nbe(q)$ satisfies the identity
\begin{equation}
\label{nbe-identity-2}
\nbe(q) = - \Heaviside(-q) + \sign(q)\nbe(\abs{q}),
\end{equation}
where $\sign(q)$ is the sign of $q$.
\end{Lem}
\begin{proof}
For $q>0$ \eqref{nbe-identity-2} is the trivial identity. For $q<0$, $\nbe(q)=-1-\nbe(-q)$ is an immediate consequence of the identity \eqref{nbe-identity-1}.
\end{proof}

\begin{Lem}
\label{lemma-3}
The Matsubara sum $S_G$ can be written as
\begin{equation}
S_G = \sum\limits_{\T}
\prod\limits_{l \in\Tbar}[ 1 + \nbe_l \big(1 - \opR_l\big) ]
I_G^{\T}.
\end{equation}
\end{Lem}
\begin{proof}
The identity \eqref{nbe-identity-2} implies
\begin{equation*}
\nbe(\eps_l q_l)=-\Heaviside(-\eps_l q_l) + \eps_l \sign(q_l)\nbe(\abs{q_l}).
\end{equation*}
Then, as we computed already, for $q_l>0$ and an arbitrary function $f(q)$,
\begin{equation*}
\big(1 - \opR_l \big)\left[-\Heaviside(-\eps_l q_l)f(q_l)\right]= \eps_l f(-\eps_l q_l),
\end{equation*}
whereas
\begin{align*}
\big(1 - \opR_l \big)\left[\eps_l \sign(q_l)\nbe(\abs{q_l})f(q_l)\right]&=
\eps_l \nbe(\abs{q_l}) \left[ \sign(q_l)f(q_l) - \sign(-q_l)f(-q_l) \right]\cr
&= \eps_l \sign(q_l) \nbe(\abs{q_l})\big(1 + \opR_l \big)f(q_l).
\end{align*}
But
\begin{equation*}
\big(1 + \opR_l \big)f(-\eps_l q_l)= \big(1 + \opR_l \big)f(q_l).
\end{equation*}
Therefore, for $q_l>0$,
\begin{equation*}
\big(1 - \opR_l \big)\nbe(\eps_l q_l)f(q_l) = \eps_l \left[ 1 + \nbe_l \big(1 + \opR_l\big) \right]
f(-\eps_l q_l),
\end{equation*}
and hence
\begin{multline*}
\prod\limits_{l \in\Tbar} \big(1 - \opR_l \big)
\left(
{\prod\limits_{j \in \T} {\frac{1}{{q_j  - i\Omega _j^{\T} (N, -iq_l )}}}}
\prod\limits_{l \in\Tbar} 2\pi\eps_l \nbe(\eps_l q_l) 
\right) = \cr
(2\pi)^\NL \prod\limits_{l \in\Tbar} \left[ 1 + \nbe_l \big(1 + \opR_l\big) \right]
{\prod\limits_{j \in \T}
{\frac{1}{{q_j  - i\Omega _j^{\T} (N, i\eps_l q_l )}}}}.
\end{multline*}
The final result follows from the basic property,
\begin{equation}
\frac{1}{q}\big(1+\opR(q)\big)f(q)=\big(1-\opR(q)\big)\left[ \frac{1}{q}f(q) \right].
\end{equation}
\end{proof}

Now we are finally in a position to prove the propositions of section \ref{sec-main-theorems}:

\begin{proof}[Proof of Theorem \ref{thm-main-1}]
We know from lemma \ref{lemma-3} that
\begin{equation}
\label{eq-SG-casi}
S_G = \sum\limits_{\T}
\prod\limits_{l \in\Tbar}[ 1 + \nbe_l \big(1 - \opR_l\big) ]
I_G^{\T}.
\end{equation}
But the identity
\begin{equation*}
\big(1 - \opR(q)\big)\frac{1}{q}\big(1 - \opR(q)\big)=\frac{1}{q}\big(1 + \opR(q)\big)\big(1 - \opR(q)\big)\equiv 0,
\end{equation*}
implies that all the operators $\big(1 - \opR_j\big)$, with $j \in \T$, annihilate each of the $I_G^\T$ defined in \eqref{IG-tree-explicit}. Therefore we can extend the product indices in \eqref{eq-SG-casi} from $l \in\Tbar$ to $i \in G$:
\begin{equation*}
\sum\limits_{\T}
\prod\limits_{l \in\Tbar}[ 1 + \nbe_l \big(1 - \opR_l\big) ]I_G^{\T}
=
\sum\limits_{\T}
\prod\limits_{i \in G}[ 1 + \nbe_i \big(1 - \opR_i\big) ]I_G^{\T}.
\end{equation*}
But now observe that the operator acting on $I_G^{\T}$ is $\T$-independent, so that we can transpose it with the sum over trees to get
\begin{align*}
S_G &=\prod\limits_{i \in G}[ 1 + \nbe_i \big(1 - \opR_i\big) ]\sum\limits_{\T} I_G^{\T}\cr
&=\prod\limits_{i \in G}[ 1 + \nbe_i \big(1 - \opR_i\big) ] I_G,
\end{align*}
according to \eqref{eq:IG-sum over trees}.
\end{proof}

\begin{proof}[Proof of Lemma \ref{lemma:cutsets}]
Let $C$ be a cutset of the graph $G$ and let $\T$ be an arbitrary tree of $G$.
The lines in $C$ cannot all belong to $\Tbar$ since then $C$ would not be a cutset (recall that $\T$ is a connected graph). Therefore, at least one of the lines in $C$ must belong to the tree $\T$, say $k$.
But then the operator
\begin{equation*}
\hat{\mathcal{A}}_C=\prod_{i\in C}\big(1 - \opR_i\big)
\end{equation*}
will contain the factor $\big(1 - \opR_k\big)$, which annihilates $I_G^\T$ in \eqref{IG-tree-explicit}. Since this will be true for any tree, the result follows.
\end{proof}

\begin{proof}[Proof of Theorem \ref{thm-main-2}]
The expansion of the product defining the operator $\TOp_G$ generates an expression similar to \eqref{eq:thermal-operator-form2}, but where the sums run over all possible tuples of lines at each order. But according to lemma \ref{lemma:cutsets}, if a tuple $C$ is a cutset of $G$, then the corresponding term will contain the operator $\hat{\mathcal{A}}_{C}$, defined above, which annihilates the integral $I_G$. So the tuples corresponding to cutsets can safely be omitted from $\TOp_G$.
Finally, since a tree of $G$ has $\NV-1$ lines, then the maximum number of lines that we can remove without disconnecting the graph $G$ is $\NL=\NI-(\NV-1)$. So all tuples with more than $L$ lines will be cutsets and hence the expansion of the product defining the operator $\TOp_G$ ends at degree $L$.
\end{proof}

\section{Example application: the calculation of $S_{G_4}$}
\label{sec-calculation of S4}

According to the results presented in this paper, we can compute the sum $S_{G_4}$ defined in \eqref{def-sumG4} by first computing its associated integral $I_{G_4}$ and then acting on it with the operator $\TOpp_{G_4}$.
According to theorem \ref{thm-main-2}, this operator is given by
\begin{equation}
\label{eq-opG4}
\begin{split}
\TOpp_{G_4}&=1+\nbe_1 \big(1-\opR_1\big)+\nbe_2 \big(1-\opR_2\big)\cr
&\qquad +\nbe_3 \big(1-\opR_3\big)+\nbe_4 \big(1-\opR_4\big)+\nbe_5 \big(1-\opR_5\big)\cr
&\qquad +\nbe_2 \nbe_4 \big(1-\opR_2\big) \big(1-\opR_4\big)
+\nbe_2 \nbe_3 \big(1-\opR_2\big) \big(1-\opR_3\big)
\cr
&\qquad +\nbe_1 \nbe_3 \big(1-\opR_1\big) \big(1-\opR_3\big)
+\nbe_1 \nbe_4 \big(1-\opR_1\big) \big(1-\opR_4\big)
\cr
&\qquad +\nbe_4 \nbe_5 \big(1-\opR_4\big) \big(1-\opR_5\big)
+\nbe_3 \nbe_5 \big(1-\opR_3\big) \big(1-\opR_5\big)
\cr
&\qquad +\nbe_2 \nbe_5 \big(1-\opR_2\big) \big(1-\opR_5\big)
+\nbe_1 \nbe_5 \big(1-\opR_1\big) \big(1-\opR_5\big),
\end{split}
\end{equation}
where, as before, $\nbe_i \equiv \nbe(q_i)$ and $\opR_i \equiv \opR(q_i)$.
\medskip

We note that the operator $\TOpp_{G_4}$ ends at degree 2, since the removal of 3 or more lines from $G_4$ disconnects the graph. Moreover, each of the quadratic terms in \eqref{eq-opG4} is associated with one of the trees of $G_4$ shown in figure \ref{fig-G4trees}). Note that there are no quadratic terms in \eqref{eq-opG4} with index combinations $\langle 1,2\rangle$ and $\langle 3,4\rangle$, since these are cutsets of $G_4$ (see figure \ref{fig-G4} for the labeling).
\medskip

The integral $I_{G_4}$ associated to $G_4$ is given by (see \eqref{eq-sumG4})
\begin{multline}
\label{def-I4}
I_{G_4}:=
\int_{-\infty}^{\infty}dx\int_{-\infty}^{\infty}dy
\left\{
{\frac{1}{{(x^2  + q_1^2 )((N_a  - x)^2  + q_2^2 )}}} \right.\\
\times
\left.
\frac{1}{((N_b  + x + y)^2  + q_3^2 )((N_b  + N_c  + x + y)^2  + q_4^2 )(y^2  + q_5^2 )} \right\}.
\end{multline}
This integral can be calculated following Gaudin's approach explained in section \ref{sec-explicit evaluations} or by directly performing the $x$ and $y$ integrations, one after the other, as we do now. Again, it is convenient to work with linear rather than quadratic denominators by expressing
\begin{equation}
\frac{1}{x^2+q^2}=\frac{1}{2q}\sum_{\varepsilon=\pm 1}\frac{\varepsilon}{ix+\varepsilon q},
\end{equation}
so that
\begin{multline*}
I_{G_4}=
\frac{1}{2q_1 2q_2 2q_3 2q_4 2q_5}\sum_{\varepsilon_1,\ldots \varepsilon_5 = \pm 1}
\int_{-\infty}^{\infty}dx
\left\{
{\frac{\varepsilon_1}{(ix  + \varepsilon_1 q_1)}}
{\frac{\varepsilon_2}{(i(N_a - x)  + \varepsilon_2 q_2)}} 
\right.\\
\times\int_{-\infty}^{\infty}dy
\left.
{\frac{\varepsilon_3}{(i(N_b + x + y) + \varepsilon_3 q_3)}}
{\frac{\varepsilon_4}{(i(N_b + N_c + x + y) + \varepsilon_4 q_4)}}
{\frac{\varepsilon_5}{(iy  + \varepsilon_5 q_5)}}
\right\}.
\end{multline*}
\medskip

Using Mathematica 7 to perform first the $y$-integral and then the $x$-integral we find

{\allowdisplaybreaks
\begin{multline}
\label{eq-I4-result}
I_{G_4} = \frac{(2\pi)^2}{2q_1 2q_2 2q_3 2q_4 2q_5}\times\\
\left[
\frac{1}{(i {N_a}+{q_1}+{q_2}) (i
   ({N_a}+{N_b})+{q_2}+{q_3}+{q_5}) (i
   ({N_a}+{N_b}+{N_c})+{q_2}+{q_4}+{q_5})} \right.\\
   +\frac{1}{(i
   {N_b}+{q_1}+{q_3}+{q_5}) (i ({N_a}+{N_b})+{q_2}+{q_3}+{q_5})
   (i ({N_a}+{N_b}+{N_c})+{q_2}+{q_4}+{q_5})}\\ 
   -\frac{1}{(i
   {N_a}-{q_1}-{q_2}) (i {N_b}+{q_1}+{q_3}+{q_5}) (i
   ({N_b}+{N_c})+{q_1}+{q_4}+{q_5})}\\ 
   +\frac{1}{(i
   {N_b}+{q_1}+{q_3}+{q_5}) (i ({N_b}+{N_c})+{q_1}+{q_4}+{q_5})
   (i ({N_a}+{N_b}+{N_c})+{q_2}+{q_4}+{q_5})}\\ 
-\frac{1}{(i {N_a}-{q_1}-{q_2}) (i {N_c}-{q_3}-{q_4}) (i
   ({N_a}+{N_b}+{N_c})-{q_2}-{q_4}-{q_5})}\\ 
   +\frac{1}{(i
   {N_a}+{q_1}+{q_2}) (i {N_c}-{q_3}-{q_4}) (i
   ({N_b}+{N_c})-{q_1}-{q_4}-{q_5})}\\ 
   -\frac{1}{(i
   {N_c}-{q_3}-{q_4}) (i ({N_b}+{N_c})-{q_1}-{q_4}-{q_5}) (i
   ({N_a}+{N_b}+{N_c})-{q_2}-{q_4}-{q_5})}\\ 
-\frac{1}{(i {N_a}+{q_1}+{q_2}) (i {N_c}-{q_3}-{q_4}) (i
   ({N_a}+{N_b})+{q_2}+{q_3}+{q_5})}\\ 
 +\frac{1}{(i
   {N_a}-{q_1}-{q_2}) (i {N_c}-{q_3}-{q_4}) (i
   {N_b}+{q_1}+{q_3}+{q_5})}\\ 
   -\frac{1}{(i {N_c}-{q_3}-{q_4}) (i
   {N_b}+{q_1}+{q_3}+{q_5}) (i
   ({N_a}+{N_b})+{q_2}+{q_3}+{q_5})}\\
+\frac{1}{(i {N_a}-{q_1}-{q_2}) (i {N_c}+{q_3}+{q_4}) (i
   ({N_a}+{N_b})-{q_2}-{q_3}-{q_5})}\\ 
   -\frac{1}{(i
   {N_a}+{q_1}+{q_2}) (i {N_c}+{q_3}+{q_4}) (i
   {N_b}-{q_1}-{q_3}-{q_5})}\\ 
   +\frac{1}{(i {N_c}+{q_3}+{q_4}) (i
   {N_b}-{q_1}-{q_3}-{q_5}) (i
   ({N_a}+{N_b})-{q_2}-{q_3}-{q_5})}\\
+\frac{1}{(i {N_a}+{q_1}+{q_2}) (i {N_c}+{q_3}+{q_4}) (i
   ({N_a}+{N_b}+{N_c})+{q_2}+{q_4}+{q_5})}\\ 
   -\frac{1}{(i
   {N_a}-{q_1}-{q_2}) (i {N_c}+{q_3}+{q_4}) (i
   ({N_b}+{N_c})+{q_1}+{q_4}+{q_5})}\\ 
   +\frac{1}{(i
   {N_c}+{q_3}+{q_4}) (i ({N_b}+{N_c})+{q_1}+{q_4}+{q_5}) (i
   ({N_a}+{N_b}+{N_c})+{q_2}+{q_4}+{q_5})}\\ 
-\frac{1}{(i {N_a}-{q_1}-{q_2}) (i
   ({N_a}+{N_b})-{q_2}-{q_3}-{q_5}) (i
   ({N_a}+{N_b}+{N_c})-{q_2}-{q_4}-{q_5})}\\ 
   -\frac{1}{(i
   {N_b}-{q_1}-{q_3}-{q_5}) (i ({N_a}+{N_b})-{q_2}-{q_3}-{q_5})
   (i ({N_a}+{N_b}+{N_c})-{q_2}-{q_4}-{q_5})}\\ 
   +\frac{1}{(i
   {N_a}+{q_1}+{q_2}) (i {N_b}-{q_1}-{q_3}-{q_5}) (i
   ({N_b}+{N_c})-{q_1}-{q_4}-{q_5})}\\   
\left.
-\frac{1}{(i
   {N_b}-{q_1}-{q_3}-{q_5}) (i ({N_b}+{N_c})-{q_1}-{q_4}-{q_5})
   (i ({N_a}+{N_b}+{N_c})-{q_2}-{q_4}-{q_5})}
\right].
\end{multline}}%

\medskip

The explicit evaluation of sum $S_{G_4}$ can be obtained from the application of the operator $\TOpp_{G_4}$ given by \eqref{eq-opG4} to the result \eqref{eq-I4-result} for the integral $I_{G_4}$. The resulting expression would fill several pages of this journal, but it can be easily be generated by a symbolic manipulation program such as Mathematica.
\bigskip\bigskip\bigskip\bigskip

\no
{\bf Acknowledgments}. The author would like to thank the hospitality of the
Center for Astronomy and Particle Theory at the University of Nottingham, where this work was written, and the financial support of Fondecyt, under grant 1070505.
The diagrams presented in this paper were produced with JaxoDraw 2.0
\cite{jaxodraw}.

\end{document}